\documentclass[11pt]{article}
\usepackage{bcc27}

\bcctitle{Expanders -- how to find them, and what to find in them}
\bccshorttitle{Expanders -- how to find them, and what to find in them}

\bccname{Michael Krivelevich\footnote{The author was supported by the Clay Mathematics Institute. In particular, this survey was written as part of the author's role as the Clay Lecturer at the 27th British Combinatorial Conference. The research  was also supported in
part by USA--Israel BSF grant 2014361 and by ISF grant 1261/17.}}
\bccshortname{Michael Krivelevich}

\bccaddressa{School of Mathematical Sciences,\\ Raymond and Beverly Sackler Faculty of Exact Sciences,\\ Tel Aviv University,\\ Tel Aviv, 6997801,\\ Israel}
\bccemaila{krivelev@tauex.tau.ac.il}

\newtheorem{thm}[theorem]{Theorem}
\newtheorem{propos}[theorem]{Proposition}

\newtheorem{corol}[theorem]{Corollary}
\newtheorem{coroll}[theorem]{Corollary}

{   \theoremstyle{definition}
\newtheorem{defin}[theorem]{Definition}
}

\newcommand{\whp}{{\bf whp}\ }

\newcommand{\VOL}{\mathrm{vol}}
\newcommand{\girth}{\mathrm{girth}}
\newcommand{\DIAG}{\mathrm{diag}}
\newcommand{\diam}{\mathrm{diam}}
\newcommand{\Bin}{\mathrm{Bin}}

\begin{document}
\bibliographystyle{amsplain}

\makebcctitle

\begin{abstract}
A graph $G=(V,E)$ is called an expander if every vertex subset $U$ of size up
to $|V|/2$ has an external neighborhood whose size is comparable to $|U|$. Expanders have been a subject of intensive research for more than three decades and have become one of the central notions of modern graph theory.

We first discuss the above definition of an expander and its alternatives. Then we present examples of families of expanding graphs and state basic properties of expanders. Next, we introduce a way to argue that a given graph contains a large expanding subgraph. Finally we  research properties of expanding graphs, such as existence of small separators, of cycles (including cycle lengths),  and embedding of large minors.
\end{abstract}

\section{Introduction}\label{s-intro}
Putting it somewhat informally, a graph $G$ is an expander if every vertex subset $U$ expands outside substantially, i.e.,  has an external neighborhood whose size is comparable to $|U|$. Since their introduction in the seventies, expanders have become one of the most central, and also one of the most applicable, notions of modern combinatorics. Their uses are numerous and wide-ranging, from extremal problems in graph theory and explicit constructions of graphs with desired properties to design and implementation of lean yet reliable communication networks. The reader is encouraged to consult the comprehensive and very readable survey of Hoory, Linial and Wigderson~\cite{HLW06}, devoted entirely to expanding graphs and covering many aspects of this subject; the very concise presentation of Sarnak~\cite{Sar04} would be a (somewhat\ldots) shorter introduction.

Ever since the seminal paper of Pinsker~\cite{Pin73}, randomness has been used to argue about expanders and their existence. In many cases, random graphs, sometimes after minor alterations, serve as typical examples of expanders with desired parameters. In fact, expanders are frequently viewed as synonymous to pseudo-random graphs, the latter being graphs whose edge distribution, and other properties, resemble closely those of truly random graphs with the same density. Pseudo-random graphs are covered in survey~\cite{KS06}. The notions of expanders and of pseudo-random graphs,  while being close, are quite distinct --- for example, a bipartite graph can be an excellent expander, but cannot quite model a random graph. Ideology and tools can be rather similar; in particular, spectral methods are used frequently to argue about both expansion and pseudo-randomness.

In this survey we undertake a (more) qualitative study of expansion and expanding graphs. Instead of looking for best possible quantitative results on expansion, strongest expanders possible, largest eigenvalue ratio etc., we aim here to study inherent properties of expanders, their derivation and consequences. We will rely less on the strength of expansion to derive our results; in a sense our assumption about expansion is fairly weak (we will sometimes even call our expanders weak explicitly). As it turns out, even the rather relaxed definition of expanders we embrace is deep enough to allow for meaningful study and quite interesting consequences.

The text is organized as follows. In the next section, we introduce a formal definition of expanders we use in this survey, and compare it with alternative but similar definitions. In Section~\ref{s2} we  discuss basic properties of expanding graphs. Examples of families of expanding graphs are presented in Section~\ref{s-examples}; we use the examples presented also to adjust our expectations, and to see which graph structures can and cannot be hoped to be found in weakly expanding graphs. The connection between expanders and separators is presented in Section~\ref{s-separators}. In Section~\ref{s-subgraphs} we introduce a general tool for proving the existence of a large expanding subgraph in a given graph, this tool is based on the notion of local sparseness; we discuss this notion in the context of sparse random graphs. Amongst the structures one can hope to find in expanders are long paths and cycles, and in Section~\ref{s-long_pc} we show that indeed weak expanders are rich in cycles,  in several well defined quantitative senses. Finally, in Section~\ref{s-minors} we present results about embedding large minors in expanding graphs.

\section{Definition(s) of an expander}\label{s-defs}
Let us start with introducing (the rather standard) notation used throughout this survey.
For a graph $G=(V,E)$ and a vertex set $U\subset V$ we denote by $N_G(U)$ the external neighborhood of $U$ in $G$, i.e.,
$$
N_G(U)=\{v\in V\setminus U: v\mbox{ has a neighbor in }U\}\,.
$$
We also write, for two disjoint subsets $U,W\subset V$,
$$
N_G(U,W)=\{v\in W: v\mbox{ has a neighbor in }U\}\,,
$$
thus $N_G(U)=N_G(U,V\setminus U)$. Also, $e_G(U,W)$ stands for the number of edges of $G$ between $U$ and $W$. Sometimes when the relevant graph $G$ is clear from the context, we will omit the subscript in the above notation.

We can now give the most basic definition of this survey, that of an $\alpha$-expander.
\begin{defin}\label{def-expander}
Let $G=(V,E)$ be a graph on $n$ vertices, and let $\alpha>0$. The graph $G$ is an {\em $\alpha$-expander} if $|N_G(U)|\ge \alpha |U|$ for every vertex subset $U\subset V$ with $|U|\le n/2$.
\end{defin}

It is important to note that although the above definition appears to cover only expansion of sets $W$ up to half the order of $G$, it extends further to sets whose size exceeds $n/2$. Indeed, if $|U|>n/2$, then the set $W=V-(U\cup N_G(U))$ has size $|W|<n/2$, and satisfies $N_G(W)\subseteq N_G(U)$. This allows to bound $|N_G(U)|$ from below as a function of $|U|$, when assuming that $G$ is an $\alpha$-expander.

Our primary concern in this survey is neither getting strongest or best possible expanders, nor pushing the quantitative notion of expansion to the limit. Our goal is rather different --- we mostly assume that we have a decent (one could even say weak) expander $G$, with $\alpha$ in the definition of an $\alpha$-expander being a small constant, say, $\alpha=0.01$, and aim to derive some nice properties of $G$. In principle, the expansion factor $\alpha$ in the above definition can even be allowed to be a vanishing function of the order $n$ of $G$, however here we mostly stick to the assumption $\alpha=\Theta(1)$. We should add here that throughout this survey we allow ourselves to use expressions such as ``weak expander", ``strong expander", etc. rather informally, aiming to indicate the extent of expansion of the corresponding graph; of course formal definitions such as $\alpha$-expansion should be used to measure the expansion formally and quantitatively.

The choice to stop the expansion assumption at $n/2$ in Definition~\ref{def-expander} may appear somewhat arbitrary. As we will see shortly, this is certainly a convenient compromise, as it allows us to derive easily some basic properties of $G$, for example its connectedness. It is quite natural though to give also more general definitions, like the ones below.
\begin{defin}\label{def-expander2}
Let $G=(V,E)$ be a graph, let $I$ be a set of positive integers, and let $\alpha>0$. The graph $G$ is an {\em $(I,\alpha)$-expander} if $|N_G(U)|\ge \alpha |U|$ for every vertex subset $U\subset V$ satisfying $|U|\in I$.
\end{defin}
\begin{defin}\label{def-expander3}
Let $G=(V,E)$ be a graph, let $k$ be a positive integer, and let $\alpha>0$. The graph $G$ is a {\em $(k,\alpha)$-expander} if $|N_G(U)|\ge \alpha |U|$ for every vertex subset $U\subset V$ satisfying $|U|\le k$.
\end{defin}
Notice, mostly to exercise the above defined notions, that an $\alpha$-expander $G$ on $n$ vertices is (exactly) an $(I,\alpha)$-expander for $I=\{1,\ldots,n/2\}$, and a $(k,\alpha)$-expander for $k=n/2$\footnote{Here and later we allow ourselves to treat rounding issues somewhat casually.}.

Are the above three notions of expansion radically different? Does an $(I,\alpha)$-expander on $n$ vertices with $I$ being just a single value, say, $I=\{n/100\}$ have properties much different from an $(n/300,\alpha)$-expander, or from an $\alpha$-expander? Well, yes in some senses (for example, the former two are easily seen not to guarantee connectedness, while the latter one does). However, since in this survey we will mostly research qualitative properties and will frequently settle for finding substructures in $G$ of size linear in $n$, without trying too much to optimize the constants involved, in many other senses the above three definitions are morally equivalent. We will provide a formal statement supporting this paradigm shortly.

Definition~\ref{def-expander} is a fairly commonly used notion of an expander, see, e.g.,~\cite{Alo86}, or~\cite[Chapter 9]{AS}. For the case of $d$-regular graphs with constant $d$, or more generally, for graphs with bounded maximum degree, being an $\alpha$-expander is essentially equivalent (qualitatively) to having edge expansion bounded away from 0. The latter notion can be quantified through the so-called Cheeger constant of a graph $G$.  For a graph $G=(V,E)$ on $n$ vertices and a subset $U\subseteq V$ set $\VOL_G(U)=\sum_{v\in U}\deg_G(v)$.
Define now
$$
h(G)= \min_{\emptyset\ne U\subsetneq V}\, \frac{e_G(U,V\setminus U)}{\min(\VOL_G(U),\VOL_G(V\setminus U))}\,,
$$
the quantity $h(G)$ is called the {\em Cheeger constant} of $G$. (See, e.g.,~\cite{Chung-book} for a general discussion.) Having $h(G)$ bounded away from 0 means the graph is a decent edge expander, and --- assuming its maximum degree is bounded  --- is a decent vertex expander as well.

Let us now address the quantitative comparison between Definitions~\ref{def-expander},~\ref{def-expander2} and~\ref{def-expander3}. As announced, we will see that they are essentially equivalent.

\begin{lemma}\label{lem1}
Let $G$ be an $(I,\alpha)$-expander on $n$ vertices with $I=\left[\frac{n}{3}, \frac{2n}{3}\right]$. Then there is a vertex subset $Z\subset V(G)$ of size $|Z|< n/3$ such that the graph $G'=G[V\setminus Z]$ is an $\alpha$-expander.
\end{lemma}

\begin{proof}
Start with $G'=G$, and as long as there is a vertex subset $A\subset V(G')$ of size $|A|\le |V(G')|/2$ satisfying $|N_{G'}(A)|<\alpha |A|$, delete $A$ and update $G':=G'[V(G')\setminus A]$. Observe that the disjoint union of the deleted sets $Z$ obviously satisfies $|N_G(Z)|<\alpha |Z|$. Hence by our assumption on the expansion of $G$ we derive that the size of $Z$ is never in $I$ during the execution of the procedure, so either $|Z|<n/3$ or $|Z|> 2n/3$. Suppose that after some iteration the size of $Z$ exceeds $2n/3$ for the first time. Let $A$ be the set deleted at this iteration, and let $Z'=Z\setminus A$. Then $|Z'|<n/3$, and $|A|\le (n-|Z'|)/2$, implying $|Z|=|Z'|+|A|\le |Z'|+(n-|Z'|)/2= n/2+|Z'|/2<n/2+n/6=2n/3$ -- a contradiction. This indicates that the deletion procedure halts with $|Z|<n/3$, and the subgraph $G[V\setminus Z]$ has the required expansion property.
\end{proof}

\begin{lemma}\label{lem2}
Let $G$ be a $(k,\alpha)$-expander. Then $G$ is a $\left(\frac{3k}{2},\frac{\alpha}{6}\right)$-expander, or there is a subset $V_0\subset V(G)$ of size $|V_0|\ge \frac{2k}{3}$  such that the induced subgraph $G[V_0]$ is an $(\alpha/2)$-expander.
\end{lemma}

\begin{proof}
We can assume that $G$ has more than $2k$ vertices, as otherwise the second alternative of the lemma is satisfied with $V_0=V(G)$. If $G$ does not satisfy the first alternative, then there exists a set $X$, $k<|X|\le \frac{3k}{2}$ with $|N_G(X)|<\frac{\alpha |X|}{6}$. But then every subset $Y\subset X$ with $\frac{|X|}{3}\le |Y|\le \frac{2|X|}{3}$, satisfies:
$$
|N_G(Y,X\setminus Y)|\ge |N_G(Y)|-|N_G(X)| \ge \alpha |Y|-\frac{\alpha |X|}{6}\ge \alpha|Y|-\frac{\alpha|Y|}{2}=\frac{\alpha|Y|}{2}\,.
$$
Feeding the induced subgraph $G[X]$ to Lemma~\ref{lem1} produces the required $(\alpha/2)$-expander on at least $2|X|/3\ge 2k/3$ vertices.
\end{proof}

\begin{lemma}\label{lem3}
Let $G=(V,E)$ be a $(\{k\},\alpha)$-expander with $0<\alpha\le 1$. Then there is a vertex subset $Z\subset V$ of size $|Z|< k$ such that the graph $G'=G[V\setminus Z]$ is an $\left(\frac{\alpha k}{3},\frac{\alpha}{3}\right)$-expander.
\end{lemma}

\begin{proof}
Start with $G'=G$, and as long as there is a vertex subset $A\subset V(G')$ of size $|A|\le \frac{\alpha k}{3}$ with $|N_{G'}(A)|<\alpha |A|/3$, delete $A$ and update $G':=G'[V(G')\setminus A]$. Observe that the union $Z$ of deleted sets obviously satisfies  $|N_G(Z)|<\alpha |Z|/3$. Assume $|Z|$ reaches $k$ after some iteration.  Let $A$ be the set deleted at this iteration, and let $Z'=Z-A$. Then $|Z'|<k$, and $|A|\le \alpha k/3$ with $|N_G(A,V-Z')|<\alpha |A|/3$. Choose an arbitrary subset $A'\subset A$ so that $|Z'\cup A'|=k$. Then
\begin{align*}
\alpha k&\le |N_G(Z'\cup A')|\le |N_G(Z')|+|N_G(A',V\setminus Z')|\le |N_G(Z')|+|N_G(A,V\setminus Z')|+|A\setminus A'|\\
&<\frac{\alpha k}{3}+\frac{\alpha |A|}{3}+|A|
<\frac{\alpha k}{3}+\frac{\alpha^2 k}{3}+\frac{\alpha k}{3}\le \alpha k
\end{align*}
-- a contradiction.
\end{proof}

We arrive at the following conclusion:
\begin{thm}\label{t-defs}
For every $0<c,\alpha\le 1$ there are $c',\alpha'>0$ such that the following holds. Let $G$ be a $(\{k\},\alpha)$-expander on $n$ vertices for $k=\lfloor cn\rfloor$. Then $G$ contains an induced subgraph $G'$ on at least $c'n$ vertices which is an $\alpha'$-expander.
\end{thm}

The above theorem indicates clearly that the alternative definitions of expanders, as given by Definitions~\ref{def-expander}--\ref{def-expander3}, are basically equivalent on the qualitative level: assuming a one-point expansion of Definition~\ref{def-expander2} with $I=\{k\}$ for $k=\Theta(n)$ delivers an $\alpha$-expander of Definition~\ref{def-expander} on linearly many vertices. Hence in the rest of this survey we will mostly work with Definition~\ref{def-expander}.

Observing Theorem~\ref{t-defs}, one may wonder whether assuming expansion of sublinear sets in a graph $G$, or even more ambitiously of sets of size up to $k$, where $k$ is much smaller than $n=|V(G)|$, can guarantee the existence of a substantially sized expanding subgraph $G'$ of $G$. For one, a $(k,\alpha)$-expander on $n$ vertices can be just a collection of disjoint cliques of size about $(1+\alpha)k$ each, and so if $k=o(n)$, such $G$ does not contain a linearly sized expanding subgraph. Thus, the assumption $k=\Theta(n)$ in Theorem~\ref{t-defs} is essential. A much more striking example is provided by Moshkovitz and Shapira~\cite{MS18}. They prove  (see Theorem 1 of~\cite{MS18}) that for every $n$ and  $k=O(\log\log n)$, there is an $n$-vertex graph $G$ with all degrees of order $k$ and $\girth(G)=\Omega(\log n/k^2)$, in which every induced subgraph $H$ on $t\ge \girth(G)$ vertices has a set of edges $E_0\subseteq E(H)$ of size $|E_0|=o(t)$, whose removal from $H$ creates a graph $H'$ with all connected components having at most $2t/3$ vertices. (Rephrasing it, every induced subgraph $H$ of $G$ with less than $\girth(G)$ vertices is a tree, while every induced $H$ with at least $\girth(G)$ vertices has a small edge-separator and thus cannot be an expander.) Such a graph $G$ can easily be argued to be a good local expander; for example, every set $U\subset V(G)$ of size $|U|\le (ck)^{(\girth(G)-1)/2}$ satisfies $|N_G(U)|\ge 2|U|$, for an appropriately chosen constant $c>0$, see Lemma 2.1 of~\cite{SV08}. Hence local expansion in a graph cannot always be traded for global expansion
in a subgraph.

\section{Basic properties of $\alpha$-expanders}\label{s2}
This short section discusses the most basic properties of $\alpha$-expanders. It can serve as a vindication of our choice to stick to Definition~\ref{def-expander} as the definition of choice for expanders.

For a graph $G=(V,E)$, a vertex $v\in V$ and a natural number $t$, denote by $B_t(v)$ the ball of radius $t$ around $v$ in $G$, i.e., $B_t(v)$ is the set of vertices of $G$ at distance at most $t$ from $v$.

\begin{propos}\label{p1}
Let $G=(V,E)$ be an $\alpha$-expander on $n$ vertices. Then for every $v\in V$ and every natural $t$, we have: $|B_t(v)|\ge \min\left\{\frac{n}{2},(1+\alpha)^t\right\}$.
\end{propos}

\begin{proof} By induction on $t$. Obvious for $t=0$. For the induction step, observe that $B_t(v)\supseteq B_{t-1}(v)$. If $|B_{t-1}(v)|\ge \frac{n}{2}$, we are done. Otherwise, by the induction hypothesis we have $|B_{t-1}|\ge (1+\alpha)^{t-1}$. Denoting $U=B_{t-1}(v)$, we have: $B_{t}(v)=U\cup N(U)$, and $|N(U)|\ge\alpha |U|$. It follows that $|B_t(v)|\ge (1+\alpha)|U|\ge (1+\alpha)^t$, as required.
\end{proof}

 \begin{coroll}\label{coroll2}
Let $G$ be an $\alpha$-expander on $n$ vertices. Then the diameter $diam(G)$ satisfies:
 $\diam(G)\le \left\lceil\frac{2(\log n-1)}{\log(1+\alpha)}\right\rceil+1=O_{\alpha}(\log n)$.
\end{coroll}

\begin{proof} Let $u\ne v\in V(G)$. Growing balls around $u,v$, we obtain by Proposition~\ref{p1} that for
 $t=\left\lceil\frac{\log (n/2)}{\log(1+\alpha)}\right\rceil$ the balls $B_t(u), B_t(v)$ have at least $n/2$ vertices each. Then they intersect each other, or are connected by an edge.
\end{proof}

In particular, an $\alpha$-expander $G$ is connected.

Turning now to the Cheeger constant, observe that an $\alpha$-expander $G$ of maximum degree $\Delta=\Delta(G)$ satisfies trivially: $h(G)\ge \frac{\alpha}{\Delta}$. Thus, assuming that both $\alpha,\Delta=\Theta(1)$, we derive that in this case the Cheeger constant of $G$ is bounded away from zero.

Going from the Cheeger constant to random walks, we derive, using the classical result of Jerrum and Sinclair~\cite{JS88}, that the mixing time $T_{mix}$ of a (lazy) random walk on an $\alpha$-expander $G$ on $n$ vertices with $\alpha=\Theta(1)$ and bounded maximum degree satisfies $T_{mix}=O(\log n)$.
The assumption that the maximum degree of $G$ is bounded is essential here, as for example taking $G$ to be two disjoint cliques of size $n/2$ connected by a nice bounded degree expander, we observe that $G$ is an $\alpha$-expander with $\alpha=\Theta(1)$, and yet it takes a random walk a linear time in $n$ on average to cross from one clique to the other. Since random walks are essentially outside the scope of this survey, we will not dwell on their aspects anymore, instead referring the interested reader to standard sources on random walks (say,~\cite{LPW09}) for background, definitions and discussion.

We can already conclude that the definition of $\alpha$-expanders we have adopted is capable of guaranteeing easily several basic and desirable properties of the graph. We aim however to set our goals much higher than that, and to derive further properties of $\alpha$-expanders, such as the existence of linearly long paths and cycles, non-existence of sublinear separators, and embedding of large minors. These properties are to be discussed later in this survey.

\section{Examples of $\alpha$-expanders}\label{s-examples}

In this section we present some archetypal (families of) examples of weak expanders. The aim here is not just to showcase concrete instances of $\alpha$-expanders, but also to adjust our level of expectations of what can be found in $\alpha$-expanders upon observing some concrete examples.

\medskip

\noindent{\bf 1. Bipartite graphs.} The complete bipartite graph $G$ with parts of size  $\alpha n$ and $(1-\alpha)n$ for $\alpha\le 1/2$ is obviously an $\alpha$-expander. This example can be sparsified by taking $G$ to be a good expander with linearly many edges between parts of size $2\alpha n$ and $(1-2\alpha)n$ (for $\alpha\le 1/4$).

This example shows us that we cannot always expect to find an odd cycle in an expander.

\medskip

\noindent{\bf 2. Spectral approach.} This is a very powerful and appealing approach, allowing to connect between expansion properties of a graph and its eigenvalues. (See, e.g.,~\cite{Chung-book,BH} for a general background on spectral graph theory.) For a graph $G$ with vertex set $V(G)=[n]$, the adjacency matrix $A(G)$ is an $n$-by-$n$ matrix whose entry $a_{ij}$ is 1 if $(i,j)\in E(G)$ and $a_{ij}=0$ otherwise. This is a symmetric matrix with $n$ real eigenvalues, usually sorted in the non-increasing order $\lambda_1\ge\lambda_2\ge\ldots\ge \lambda_n$; they are usually called the eigenvalues of the graph $G$ itself. For the simplest case of $d$-regular graphs, the first (or trivial) eigenvalue is $\lambda_1=d$, and it is the second eigenvalue that governs the graph expansion. Specifically, the following is true (see, e.g., Corollary 9.2.2 of~\cite{AS}):

\begin{propos}\label{p-eigenvalues}
Let $G$ be a $d$-regular graph on $n$ vertices with the second largest eigenvalue $\lambda$. Then $G$ is an $\alpha$-expander, with $\alpha=\frac{d-\lambda}{2d}$.
\end{propos}

Hence a $d$-regular graph with eigenvalue gap $d-\lambda_2=\Theta(1)$ and $d=O(1)$ is a weak expander. This implication is reversible -- a regular expander has its two first eigenvalues well separated, see~\cite{Alo86}.

The spectral approach extends to the general (non-regular) case.  (See~\cite{Chung-book} for an extensive discussion of the subject and all missing definitions.) Let $G$ be a graph on $n$ vertices, and let $0=\mu_0\le \mu_1\le\ldots\le \mu_{n-1}$ be the eigenvalues of the normalized Laplacian ${\cal L}(G)$ of $G$, defined by ${\cal L}=I_n-D^{-1/2}AD^{1/2}$, where $D=D(G)=\DIAG({\deg(v_1),\ldots,\deg(v_n)})$ is the degree matrix of $G$ (we use the convention $D^{-1}_{ii}=0$ when $\deg(v_i)=0$).  We set $\mu(G)=\mu_1$. Let $h(G)$ be the Cheeger constant of $G$.  The famous {\em Cheeger inequality for graphs}~\cite{Dod84, AM85, Alo86} states that
\begin{equation}\label{lb-Cheeger}
\frac{h^2(G)}{2}\le \mu(G)\le 2h(G)\,.
\end{equation}
In particular, assuming that $G$ has its degrees bounded, and recalling the prior discussion about expansion and the Cheeger constant, we conclude that having the first non-trivial eigenvalue of the normalized Laplacian $\mu(G)$ bounded away from 0 guarantees expansion. (Here we use the second part of the Cheeger inequality (\ref{lb-Cheeger}), the first (and more involved) part will be utilized later, when discussing algorithmic issues.)

\medskip

\noindent{\bf 3. Stretching edges of a bounded degree expander.} This is a systematic way to obtain a (weaker) expander from a given expander $G$ by stretching, or subdividing, the edges of $G$.
\begin{propos}\label{p-stretching}
Let $G$ be an $\alpha$-expander of maximum degree $\Delta=O(1)$, and let $\ell$ be a positive integer. Subdividing each edge of $G$ $\ell$ times produces an $\Omega_{\Delta}(\alpha/\ell)$-expander $G'$.
\end{propos}

While being a nice way of producing new expanders, the above proposition should turn all warning lights on. Indeed, even starting with a very strong bounded degree expander $G$, the new graph $G'$ has no adjacent vertices of degree more than 2. Hence, one {\em cannot} hope to embed in $G'$ any graph $H$ in which some two vertices of degree at least three are adjacent. Even if $H$ is a little tree with two vertices of degree three adjacent -- $H$ is not there in $G'$. This seems to indicate very gloomy prospects for this line of business, and for this survey in particular -- there are seemingly not many structures one can seek to embed in an $\alpha$-expander... Not all is lost however -- weak expanders do have a rich enough structure, and we can argue about the existence of long paths and cycles, and also about embedding minors. This is indeed what we plan to do in the sequel.

\medskip

\noindent{\bf 4. Random regular graphs.} Let $G_{n,d}$ be the probability space of $d$-regular graphs on $n$ vertices. (See, e.g.,~\cite{Wor99}, or~\cite[Chapter 10]{FK} for background.) For every fixed $d\ge 3$ and growing $n$ a graph $G$ drawn from $G_{n,d}$ is typically an expander. This can be shown directly using the so-called configuration model, as was done by Bollob\'as in~\cite{Bol88}; another possible approach is to invoke Proposition ~\ref{p-eigenvalues} and known results about eigenvalues of random regular graphs~\cite{Fri08}.

\section{Expanders and separators}\label{s-separators}
Separators are one of the central concepts in graph theory in general, and in structural graph theory in particular.

\begin{defin}\label{def-separator}
Let $G=(V,E)$ be a graph on $n$ vertices. A vertex set $S\subset V$ is a {\em separator} in $G$ if there is a partition $V=A\cup B\cup S$ of the vertex set of $G$ such that $|A|,|B|\le 2n/3$, and $G$ has no edges between $A$ and $B$.
\end{defin}

Separators serve to measure quantitatively the connectivity of large vertex sets in graphs; the fact that all separators in $G$ are large indicates that it is costly to break $G$ into large pieces not connected by any edge. Observe that if $G$ is a bounded degree graph, then disconnecting $G$ can be done cheaply by removing $O(1)$ neighbors of any vertex $v\in V(G)$; however, if $G$ is well connected, then finding a small sized separator might be impossible.

Separators came into prominence with the celebrated result of Lipton and Tarjan~\cite{LT79}, asserting that every planar graph on $n$ vertices has a separator of size $O\left(\sqrt{n}\right)$. This line of research, advancing the alternative ``a small separator or a large minor" to be addressed later in this survey, has been quite fruitful over the years, see, e.g.,~\cite{AST90,PRS94,KR10}. We will return to it later in this survey.

It is easy to argue that expanders do not have small separators.

\begin{propos}\label{prop-sep1}
Let $G$ be an $\alpha$-expander on $n$ vertices, and let $S$ be a separator in $G$. Then  $|S|\ge \frac{\alpha n}{3(1+\alpha)}$.
\end{propos}

\begin{proof}
Let $S$ be a separator in $G$ of size $|S|=s$, separating $A$ and $B$, with $|A|=a$, $|B|=b$; we assume $a\le b\le 2n/3$. Then $a+s\ge n/3$. Clearly, $N_G(A)\subseteq S$. Since $a\le n/2$, by the definition of an $\alpha$-expander we get $s-\alpha a\ge 0$. Multiplying this inequality by $1/\alpha$ and summing with $a+s\ge n/3$, we obtain: $s(1+1/\alpha)\ge n/3$, or $s\ge \frac{\alpha n}{3(1+\alpha)}$.
\end{proof}

Perhaps more surprisingly, it turns out that the opposite implication is true in a well defined quantitative sense -- graphs without small separators contain large induced expanders.

\begin{propos}\label{prop-sep2}
Let $\beta>0$ be a constant. If all separators in a graph $G$ on $n$ vertices are of size at least $\beta n$, then $G$ contains an induced $(3\beta/2)$-expander on at least $2n/3$ vertices.
\end{propos}

The proof is somewhat similar to Lemmas~\ref{lem1} and~\ref{lem3}.

\begin{proof}
Start with $G'=G$, and delete repeatedly subsets $A$ of size at most $|V(G')|/2$ not expanding themselves into $V(G')$ by at least the factor of $3\beta/2$, each time updating $G':=G'-A$. Let $Z$ be the union of the deleted sets, clearly $|N_G(Z)|<3\beta |Z|/2$. If the size of $Z$ ever reaches $n/3$, focus on the moment it happens for the first time, let $A$ be the last set deleted, and let $Z'=Z\setminus A$. Then $|A|\le (n-|Z'|)/2$ and $|Z'|\le n/3$. Combining these two inequalities, we get: $|Z|=|Z'|+|A|\le 2n/3$. The set $N_G(Z)$ forms a separator in $G$ (separating between $Z$ and $V(G)\setminus (Z\cup N_G(Z))$), hence by the proposition's assumption its size is at least $\beta n\ge 3\beta|Z|/2$ -- a contradiction. We conclude that the deletion process stops with $|Z|<n/3$, and the final graph of the process is a $(3\beta/2)$-expander on at least $2n/3$ vertices, as required.
\end{proof}

The above arguments can easily be extended to sublinear separators and subconstant expansion.

We can learn a very important qualitative lesson here. It turns out that weak expanders and graphs without small separators are essentially the same thing (at least if we do not care that much about multiplicative constants involved). Thus, when aiming to prove results about graphs without sublinear separators, we can choose instead to operate in the realm of weak expanders. This simple yet powerful connection between two central graph theoretic notions (expanders and separators), usually perceived as belonging to quite different worlds (extremal graph theory and structural graph theory, respectively) can be quite fruitful and illuminating.

\section{Finding large expanding subgraphs}\label{s-subgraphs}
Given the prominence of expanders and their usability, it is tempting to argue that every graph is an expander, or at least contains a large expander inside. This is obviously too much wishful thinking. Not every graph is an expander, moreover essentially every standard notion of expansion is rather fragile -- adding an isolated vertex to a strong expander $G$ produces a non-expanding graph $G'$. Not every graph, even of average degree bounded away from zero, contains a substantially sized expanding subgraph. Indeed, planar graphs, or more generally graphs of bounded genus have sublinear separators~\cite{LT79, GHT84}. Hence, applying Proposition~\ref{prop-sep1}, we conclude that such graphs do not contain expanding subgraphs of super-constant size. Recall also the example of Moshkovitz and Shapira~\cite{MS18}, discussed in Section~\ref{s-defs}.

There is a partial remedy to the problem with this general approach. Shapira and Sudakov~\cite{SS15} and later Montgomery~\cite{Mon15} argued that {\em every} graph $G$ contains a very weak expander $G^*$ of nearly the same average degree. Their notion of expansion is different -- the expansion required is gradual in terms of subset sizes; for an $m$-vertex graph to be a weak expander, vertex sets of size $m^c$, $0<c<1$, should expand by a constant factor, whereas linearly sized vertex sets are required to expand only by about $1/\log m$ factor. (We are rather informal here in our descriptions, see the actual papers~\cite{SS15,Mon15} for accurate definitions.) Neither of these results guarantees a  (weakly) expanding subgraph on linearly many vertices.
Much earlier, Koml\'os and Szemer\'edi, in their work on topological cliques in graphs~\cite{KS94,KS96}, presented a fairly general scheme for arguing about the existence of weak expanders in any given graph; their scheme does not provide -- naturally -- for finding linearly sized expanders, or expanders with constant expansion of subsets.

We now present a fairly simple sufficient condition guaranteeing the existence of a large expanding subgraph in a given graph. This condition is based on the quantitative notion of local sparseness.
\begin{defin}\label{def-sparseness}
Let $c_1>c_2>1$, $0<\beta<1$. A graph $G=(V,E)$ on $n$ vertices is called a {\em $(c_1,c_2,\beta)$-graph} if
\begin{enumerate}
\item
$\frac{|E|}{|V|}\ge c_1$\,;
\item every vertex subset $W\subset V$ of size $|W|\le \beta n$ spans fewer than $c_2|W|$ edges.
\end{enumerate}
\end{defin}
In words, the above condition says that relatively small sets are sizably sparser than the whole graph. It has been used in recent papers~\cite{Kri16, Kri18} of the author.

How natural or common is this condition? As the (very easy) proposition below shows, most sparse graphs are locally sparse.

\begin{propos}\label{p-sparse1}
Let $c_1>c_2>1$ be reals. Define $\beta=\left(\frac{c_2}{5c_1}\right)^{\frac{c_2}{c_2-1}}$. Let $G$ be a random graph drawn from the probability distribution $G\left(n,\frac{c_1}{n}\right)$. Then with high probability every set of $k\le \beta n$ vertices of $G$ spans fewer than $c_2k$ edges.
\end{propos}

\begin{proof}
The probability in $G(n,p=c_1/n)$ that there exists a vertex subset violating the required property is at most
\begin{eqnarray*}
&&\sum_{i\le \beta n}\binom{n}{i}\binom{\binom{i}{2}}{c_2i}\cdot p^{c_2i}\le
 \sum_{i\le \beta n}\left(\frac{en}{i}\right)^i\cdot \left(\frac{eip}{2c_2}\right)^{c_2i}
=\sum_{i\le\beta n}\left[\frac{en}{i}\cdot\left(\frac{ec_1i}{2c_2n}\right)^{c_2}\right]^i\\
&=&\sum_{i\le\beta n}\left[\frac{e^{c_2+1}c_1^{c_2}}{(2c_2)^{c_2}}\cdot\left(\frac{i}{n}\right)^{c_2-1}\right]^i\,.
\end{eqnarray*}

Denote the $i$-th summand of the last sum by $a_i$. Then, if $i\le n^{1/2}$ we get: $a_i\le \left(O(1)n^{-\frac{c_2-1}{2}}\right)^i$, implying $\sum_{1\le i\le n^{1/2}}a_i=o(1)$. For $n^{1/2}\le i\le \beta n$, we have, recalling the expression for $\beta$:
$$
a_i\le \left[\frac{e^{c_2+1}c_1^{c_2}}{(2c_2)^{c_2}}\cdot\left(\frac{c_2}{5c_1}\right)^{c_2}\right]^i= \left(e\cdot\left(\frac{e}{10}\right)^{c_2}\right)^i=o(n^{-1})\,.
$$
It follows that $\sum_{i\le\beta n}a_i=o(1)$, and the desired property of the random graphs holds with high probability.
\end{proof}

One can also cap easily the typical maximum degree in (a nearly spanning subgraph of) a random graph.

\begin{propos}\label{p-sparse2}
For every $C>0$ and all sufficiently small  $\delta>0$ the following holds. Let $G$ be a random graph drawn from the probability distribution $G\left(n,\frac{C}{n}\right)$. Then with high probability every set of $\frac{\delta}{\ln\frac{1}{\delta}} n$ vertices of $G$ touches fewer than $\delta n$ edges.
\end{propos}

We omit the straightforward proof.

Observe that if $G=(V,E)$ satisfies the conclusion of the above proposition, then by deleting $\frac{\delta}{\ln\frac{1}{\delta}} n$ vertices of highest degrees in $G$, one obtains a spanning subgraph $G'=(V',E')$ on $|V'|=\left(1-\frac{\delta}{\ln\frac{1}{\delta}}\right) n$ vertices and with $|E'|\ge |E|-\delta n$ edges, in which all degrees are at most $2\ln(1/\delta)$. (Otherwise, all deleted vertices are of degree at least $2\ln(1/\delta)$, forming a subset touching at least $\delta n$ edges -- a contradiction.)


We now formulate the main result of this section.

\begin{thm}[\cite{Kri18}]\label{t-sparse}
Let $c_1>c_2>1$, $0<\beta<1$, $\Delta>0$. Let $G=(V,E)$ be a graph on  $n$ vertices, satisfying:
\begin{enumerate}
\item $\frac{|E|}{|V|}\ge c_1$\,;
\item every vertex subset $U\subset V$ of size $|U|\le \beta n$ spans fewer than $c_2|U|$ edges;
\item $\Delta(G)\le \Delta$.
\end{enumerate}
Then $G$ contains an induced subgraph $G^*=(V^*,E^*)$ on at least $\beta n$ vertices which is an $\alpha$-expander, for $\alpha=\frac{c_1-c_2}{\Delta\cdot\left\lceil\log\frac{1}{\beta}\right\rceil}$.
\end{thm}

Putting it informally, every locally sparse graph $G$ of bounded maximum degree contains a linearly sized expander $G^*$. In our terminology, the first two conditions above say precisely that $G$ is a $(c_1,c_2,\beta)$-graph. They are spelled out in full in the statement above so as to make it self-contained.

\medskip

\begin{proof}[Proof of Theorem~\ref{t-sparse} (sketch).]
The proof proceeds in rounds. We describe the first round here, the rest is fairly similar.

Choose a constant $0<\delta\ll c_1-c_2$. Let $H=(W,F)$ be a minimal by inclusion non-empty induced subgraph of $G$, satisfying $|F|/|W|\ge c_1$. (Such a subgraph exists, as $G$ itself meets this condition.) Due to the local sparseness assumption, we have $|W|>\beta n$. Also, every subset $U\subseteq W$ touches at least $c_1|U|$ edges of $H$.  Otherwise, deleting $U$ is easily seen to produce a smaller induced subgraph $H'$, still meeting the requirement stated in the definition of $H$ -- a contradiction. (We could add that $H$ is connected due to its minimality, but we do not need this property here.)

Let us analyze now the expansion properties of $H$. If a subset $U\subset W$ has at most $\beta n$ vertices, then it spans at most $c_2|U|$ edges in $G$, and thus in $H$, and yet touches at least $c_1|U|$ edges. This shows that at least $(c_1-c_2)|U|$ edges of $H$ leave $U$. Recalling the maximum degree assumption, we conclude that $U$ has at least $(c_1-c_2)|U|/\Delta$ neighbors outside of it in $H$.

At this point of the proof we have assured that $H$ is a $(\beta n, \frac{c_1-c_2}{\Delta})$-expander. We can complete the proof by appealing to the general (and somewhat inexplicit) statement of Theorem~\ref{t-defs}. We prefer however, just like in the original paper~\cite{Kri18}, to provide a self-contained argument, delivering also a better (and explicit) estimate for the order of the resulting subgraph and for its expansion strength.

Consider now medium-sized subsets $U$ in $H$. Let $\beta n<|U|\le |W|/2$. Such $U$ touches at least $c_1|U|$ edges of $U$. If every such $U$ spans at most $(c_1-\delta)|U|$ edges, we get it expanding to $\delta|U|/\Delta$ vertices outside, thus indicating that $H$ is a desired expander. Otherwise there is a subset $U$, $\beta n<|U|\le |W|/2$, spanning at least $(c_1-\delta)|U|$ edges. We can now switch our attention to the induced graph $H[U]=G[U]$, whose order is at most half that of the original graph, yet we did not give much in terms of its density, compared to the assumption on $G$ -- it is now at least $c_1-\delta$.

Iterating this argument at most $O(\log(1/\beta))$ times, we eventually arrive at the desired linearly sized $\alpha$-expander inside $G$.
\end{proof}

The assumption about bounded maximum degree $\Delta(G)=O(1)$ is used in the proof only to trade edge expansion for vertex expansion.

The proof of Theorem~\ref{t-sparse} is inherently non-algorithmic, as it uses the existence of a minimal by inclusion subgraph $H$ of prescribed density, and such a subgraph is hard to find efficiently. It turns out that with some loss in constants, one can provide an algorithmic proof of the theorem. The argument for such a proof is based on the fact that the proof of (the first part of) the Cheeger inequality (\ref{lb-Cheeger}) can be made constructive in the sense that, given a graph $G=(V,E)$ on $n$ vertices, one can find in time polynomial in $n$ a subset $U\subset V$, $\VOL_G(U)\le \VOL_G(V)/2$, satisfying $e_G(U,V\setminus U)\le \sqrt{2\mu(G)}\cdot \VOL_G(U)$. (See~\cite{Alo86}, or~\cite[Ch. 2]{Chung-book}, or~\cite{Chu10}, or~\cite[Sect. 4.5]{HLW06}.) This yields the following algorithmic result.

\begin{thm}[\cite{Kri18}]\label{t-sparse2}
Let $c_1>c_2>1$, $0<\beta<1$, $\Delta>0$. There exist a constant $\alpha=\alpha(c_1,c_2,\beta,\Delta)>0$ and an algorithm, that, given an $n$-vertex graph $G=(V,E)$ with $\Delta(G)\le \Delta$ and $|E|/|V|\ge c_1$, finds in time polynomial in $n$ a subset $U\subset V$ of size $|U|\le \beta n$, spanning at least $c_2|U|$ edges in $G$, or an induced $\alpha$-expander $G^*=(V^*,E^*)\subseteq G$ on at least $\beta n$ vertices.
\end{thm}
See Section 2 of~\cite{Kri18} for the proof.

Theorem~\ref{t-sparse} can be used to argue that a supercritical random graph $G(n,c/n)$, $c>1$, contains \whp an induced expander of linear size. This can be quite handy as it allows to extend immediately known properties of expanders to sparse random graphs.

\begin{corol}\label{cor-sparse}
For every $\epsilon>0$ there exist $\alpha,\beta>0$ such that a random graph $G\sim G\left(n,\frac{1+\epsilon}{n}\right)$ contains with high probability an induced bounded degree $\alpha$-expander on at least $\beta n$ vertices.
\end{corol}

\begin{proof}
Due to the standard monotonicity arguments we can assume that $\epsilon$ is small enough where necessary.

We will utilize several (very standard) facts about supercritical random graphs. It is known (see, e.g.,
\cite[Ch.\ 5]{JLR}) that \whp $G\sim G\left(n,\frac{1+\epsilon}{n}\right)$ contains a connected component $C_1=(V_1,E_1)$ (the so called giant component) satisfying:
\begin{eqnarray*}
|V_1|&=&2\epsilon (1+o_{\epsilon}(1))n\,,\\
\frac{|E_1|}{|V_1|} &=& 1+(1+o_{\epsilon}(1))\frac{\epsilon^2}{3}\,.
\end{eqnarray*}

Also, by Proposition~\ref{p-sparse2} \whp every $\frac{\epsilon^3}{10\ln\frac{1}{\epsilon}}n$ vertices of $C_1$ touch at most $\frac{\epsilon^3}{3}n$ edges. Deleting $\frac{\epsilon^3}{10\ln\frac{1}{\epsilon}}n$ vertices of highest degrees from  $C_1$, one gets a graph $G_0=(V_0,E_0)$ of maximum degree $\Delta(G_0)\le 7\ln\frac{1}{\epsilon}$. In addition,
\begin{eqnarray*}
|V_0|&\ge &\left(2\epsilon (1+o_{\epsilon}(1))-\frac{\epsilon^3}{10\ln\frac{1}{\epsilon}}\right)n
=2\epsilon (1+o_{\epsilon}(1))n\,,\\
|E_0|&\ge& |E_1|-\frac{\epsilon^3}{3}n
\ge |V_1|\left(1+\frac{\epsilon^2}{3}+o(\epsilon^2)\right)-\frac{\epsilon^3}{3}n\\
 &\ge& |V_0|\left(1+\frac{\epsilon^2}{3}+o(\epsilon^2)\right)-\frac{\epsilon^3}{3}n\ge |V_0|\left(1+\frac{\epsilon^2}{7}\right).
\end{eqnarray*}
Finally, applying Proposition~\ref{p-sparse1} with $c_1=1+\epsilon$, $c_2=1+\frac{\epsilon^2}{10}$, we get that \whp every $k\le \beta n$ vertices of $G_0$ (with $\beta=\beta(\epsilon)$ from Proposition~\ref{p-sparse1}) span fewer than $(1+\epsilon^2/10)k$ edges. The conditions are set to call Theorem~\ref{t-sparse} and to apply it to $G_0$; we conclude that, given the above likely events, $G_0$ contains a linearly sized $\alpha$-expander.
\end{proof}

We remark here that in order to get this qualitative result, one does not really need to apply the heavy machinery of random graphs --- it is enough actually to argue from the ``first principles". Indeed, given the likely existence of the giant component $C_1$ in the supercritical regime, one can prove easily (for example, through sprinkling) that its density is typically above 1.

Let us give yet another illustrative example of how Theorem~\ref{t-sparse} can be utilized. This example comes from the realm of positional games. (The reader can consult monograph~\cite{HKSS-book} for a general background on the subject.) In a {\em Maker--Breaker} game two players, called Maker and Breaker, claim alternately free edges of the complete graph $K_n$, with Maker moving first. Maker claims one edge at a time, while Breaker claims $b\ge 1$ edges (or all remaining fewer than $b$ edges if this is the last round of the game). The integer parameter $b$ is the so-called {\em game bias}. Maker wins the game if the graph of his edges in the end possesses a given monotone graph theoretic property, Breaker wins otherwise, with draw being impossible. In the {\em long cycle game} Maker's goal is to create as long as possible cycle. For $b\ge n/2$,  Bednarska and Pikhurko proved~\cite{BP05} that Breaker has a strategy to force Maker's graph being acyclic. On the other hand, if $b=b(n)$ is such that Maker ends up with at least $n$ edges, his graph contains a cycle (of some length) in the end. We can argue, using expanders, that already for $b=(1-\epsilon)n/2$, Maker can force a linearly long cycle in his graph.

\begin{thm}[\cite{Kri16}]\label{t-MB}
For every $\epsilon>0$ there exists $c>0$ such that for all sufficiently large $n$, when playing the $b$-biased Maker--Breaker game on $E(K_n)$ with $b\le (1-\epsilon)\frac{n}{2}$, Maker has a strategy to create a cycle of length at least $cn$.
\end{thm}

\begin{proof}[Proof (sketch)] Maker plays randomly (i.e., by taking a random free edge) during the first $(1+\epsilon/2)n$ rounds of the game. Observe that at any moment during these rounds, the board still has $\Theta(n^2)$ free edges, and thus the probability of any edge to be chosen by Maker at a given round is $O(n^{-2})$. This allows us to apply standard union-bound type calculations to claim that with positive probability Maker, playing against any strategy of Breaker, can create a graph $M_0$ on $n$ vertices with the following properties:
\begin{itemize}
\item $M_0$ has  $\left(1+\frac{\epsilon}{2}\right)n$ edges;
\item every $k\le\delta n$ vertices span at most $\left(1+\frac{\epsilon}{8}\right)k$ edges;
\item every $\delta n$ vertices touch at most $\frac{\epsilon n}{4}$ edges\,,
\end{itemize}
for some $\delta=\delta(\epsilon)>0$.
Since the game analyzed is a perfect information game with no chance moves, it follows that in fact Maker has a (deterministic) strategy to create a graph $M_0$ with the above stated properties. Take such $M_0$ and delete $\delta n$ vertices of highest degrees. The obtained graph $M_1$ has $(1-\delta)n$ vertices, at least $\left(1+\frac{\epsilon}{4}\right)n$ edges, maximum degree $\Delta(M_1)\le \frac{\epsilon}{2\delta}$, and every $k\le \delta n$ vertices span at most $\left(1+\frac{\epsilon}{8}\right)k$ edges. Applying Theorem~\ref{t-sparse} shows that such $M_1$, being a part of Maker's graph by the end of the game, contains a linearly sized $\alpha$-expander, for some $\alpha=\alpha(\epsilon)>0$. As we will argue in the next section (see Theorem~\ref{t-long_cycle}), Maker's graph contains then a linearly long cycle.
\end{proof}

\section{Long paths and cycles}\label{s-long_pc}

Now we start investigating extremal properties of expanders directly. The first objects to explore are paths and cycles. As we explained in Section~\ref{s-examples}, it is unrealistic to expect appearance of many substructures in every expanding graphs. Luckily paths and cycles (the latter ones possibly with some restrictions on their lengths) are not in this excluded category, and indeed, as we will argue in this section, every $\alpha$-expander on $n$ vertices contains cycles whose length is linear in $n$.

We first describe the most basic -- and very handy -- tool for arguing about long paths and cycles in expanding graphs, which is the Depth First Search algorithm (DFS).

\subsection{DFS algorithm}\label{ss-DFS}
The {\em Depth First Search} is a well known graph exploration algorithm, usually applied to discover connected components of an input graph. As it turns out, due to its nature (of always pushing deeper, as its name suggests), this algorithm is particularly suitable for finding long paths and cycles in graphs.
Some illustrative examples of its applications in the context of extremal graphs and random graphs include~\cite{AKS81, AKS00, BBDK06, BKS12, KS13, Rio14, KLS15}.

Recall that the DFS (Depth First Search) is a graph search algorithm
that visits all vertices of a  graph and eventually discovers the structure of the connected components of $G$.
The algorithm receives as an input a graph $G=(V,E)$; it is also assumed that an order $\sigma$ on the
vertices of $G$ is given, and the algorithm prioritizes vertices
according to $\sigma$. The algorithm's output is a spanning forest $F$ of $G$, whose connected components are identical to those of $G$. Each connected component $C$ of $F$ is a tree rooted at vertex $r$, the first vertex of $C$ according to $\sigma$.

The algorithm maintains three sets of vertices, letting
$S$ be the set of vertices whose exploration is complete, $T$ be the
set of unvisited vertices, and $U=V\setminus(S\cup T)$, where the
vertices of $U$ are kept in a stack (the last in, first out data
structure). It initializes with $S=U=\emptyset$ and
$T=V$, and runs till $U\cup T=\emptyset$. At each round of the
algorithm, if the set $U$ is non-empty, the algorithm queries $T$
for neighbors of the last vertex $v$ that has been added to $U$,
scanning $T$ according to $\sigma$. If $v$ has a neighbor $u$  in
$T$, the algorithm deletes $u$ from $T$ and inserts it into $U$. If
$v$ does not have a neighbor in $T$, then $v$ is popped out of $U$
and is moved to $S$. If $U$ is empty, the algorithm chooses the
first vertex of $T$ according to $\sigma$, deletes it from $T$ and
pushes it into $U$.

The following are basic properties of the DFS algorithm:
\begin{itemize}
\item[{\bf (P1)}] at each round of the algorithm one vertex moves, either from
$T$ to $U$, or from $U$ to $S$;
\item[{\bf (P2)}] at any stage of the algorithm, it has been revealed already
that the graph $G$ has no edges between the current set $S$ and the
current set $T$;
\item[{\bf (P3)}] the set $U$ always spans a path;
\item[{\bf (P4)}] let $F$ be a forest produced by the DFS algorithm applied to $G$. If $(u,v)\in E(G)\setminus E(F)$, then one of $\{u,v\}$ is a predecessor of the other in $F$ (along the unique path to the root).
\end{itemize}

Properties  {\bf (P1)}, {\bf (P2)} are immediate upon a brief reflection on the description of the DFS algorithm. For Property {\bf (P3)}, observe that when a vertex $u$ is added to $U$, it happens because $u$ is a neighbor of the last
vertex $v$ in $U$; thus $u$ augments the path spanned by $U$, of which $v$ is the last vertex; moving the last vertex $v$ of $U$ over to $S$ deletes the last vertex of the path spanned by $U$, and hence still leaves the updated $U$ spanning a path. Property {\bf (P4)}, applied for example in~\cite{AKS00}, is perhaps less obvious, and therefore we supply its short proof here. Since $u,v$ are connected by an edge, they obviously belong to the same connected component $C$ of $G$ (which is also a connected component of $F$, vertex-wise). Assume that $v$ was discovered by the algorithm (and moved to $U$) earlier than $u$. If $u$ is not under $v$ in the corresponding rooted tree component $T$ of $F$, then the algorithm has completed exploring $v$ and moved it over to $S$ before getting to explore $u$ (and to move it to $U$) --- obviously a contradiction.

\subsection{Long paths}\label{ss-paths}
We now exploit the features of the DFS algorithm to derive the existence of long paths in expanding graphs. The following simple to prove statement is stated explicitly in~\cite{Kri13} (see Proposition 2.1 there).

\begin{propos}\label{p-DFS}
Let $k,\ell$ be positive integers. Assume that $G=(V,E)$ is a graph on more than $k$ vertices, in which every vertex subset $S$ of size $|S|=k$ satisfies: $|N_G(S)|\ge \ell$. Then $G$ contains a path of length $\ell$.
\end{propos}

\begin{proof}
Run the DFS algorithm on $G$, with $\sigma$ being an arbitrary ordering of $V$. Look at the moment during the algorithm execution when the size of the set $S$ of already processed vertices becomes exactly equal to $k$ (there is such a moment due to Property  {\bf (P1)} above, as the vertices of $G$ move into $S$ one by one, till eventually all of them land there). By Property {\bf (P2)}, the current set $S$ has no neighbors in the current set $T$, and thus $N_G(S)\subseteq U$, implying $|U|\ge \ell$.  The last move of the algorithm was to shift a vertex from $U$ to $S$, so before this move $U$ was one vertex larger. The set $U$ always spans a path in $G$, by Property {\bf (P3)}. Hence $G$ contains a path of length $\ell$.
\end{proof}

It thus follows that if $G$ is $\alpha$-expander on $n$ vertices, then $G$ contains a path of length at least $\alpha n/2$ (apply Proposition~\ref{p-DFS} with $k=n/2$, $\ell=\alpha n/2$). As we have seen, with the right technology (DFS in this case) the proof of this result becomes essentially a one-liner. The estimate delivered by this argument is fairly tight, also in terms of the dependence on $\alpha$ -- if $G$ is taken to be the complete bipartite graph with sides of sizes $\alpha n$ and $(1-\alpha)n$, then a longest path in $G$ has $2\alpha n$ vertices.

We now describe an application of Proposition~\ref{p-DFS} to size Ramsey numbers. (See, e.g.,~\cite{Kri16} or~\cite{DP18} for the relevant background.) For graphs $G,H$ and a positive integer $r$, we write $G\rightarrow (H)_r$ if every $r$-coloring of the edges of $G$ produces a monochromatic copy of $H$. The {\em $r$-color size Ramsey number} $\hat{R}(H,r)$ is defined as the minimal possible number of edges in a graph $G$, for which $G\rightarrow (H)_r$. Let us discuss briefly  the multi-color size Ramsey number $\hat{R}(P_n,r)$ of the path $P_n$ on $n$ vertices.  Beck in his groundbreaking paper~\cite{Bec83} proved that $\hat{R}(P_n,2)=O(n)$; his argument can easily be extended to show that $\hat{R}(P_n,r)=O_r(n)$ for any fixed $r\ge 2$. If so, the relevant question now is what is the hidden dependence on $r$ in the above bound. Here we argue that $\hat{R}(P_n,r)=O(r^2\log r)n$.

The following proposition is tailored to handle size Ramsey numbers of long paths.

\begin{propos}\label{p-Ramsey}
Let $d,r>0$.  Let  $G=(V,E)$ be a graph on $|V|=N$ vertices with average degree at least $d$. Assume that for an integer $0<n<N$, every vertex subset $W\subset V$ of cardinality at most $2n$ spans at most $(d/8r)|W|$ edges. Then every edge subset $E_0\subset E$ with $|E_0|\ge |E|/r$ contains a path of length $n$.
\end{propos}

\begin{proof}
Denote $G_0=(V,E_0)$. Clearly the average degree of $G_0$ is at least $d/r$. Find a subgraph $G_1=(V_1,E_1)\subseteq G_0$ of minimum degree at least $d/2r$. Since $|E_1|\ge (d/4r)|V_1|$, we have $|V_1|>2n$ by the proposition's assumption. Let now $U\subset V_1$ be an arbitrary set of $n$ vertices in $G_1$. Then the set $U\cup N_{G_1}(U)$ spans all edges of $G_1$ touching $U$, and there are at least $|U|(d/2r)/2=nd/4r$ of them. Hence $|U\cup N_{G_1}(U)|\ge 2n$. We derive: $|N_{G_1}(U)|\ge n$. Applying Proposition~\ref{p-DFS} to $G_1$ with $k=\ell=n$, the proposition follows.
\end{proof}

Note that the assumption in the proposition is fairly similar to the definition of locally sparse graphs as in  Section~\ref{s-subgraphs}.

We can now prove
\begin{thm}\label{t-Ramsey}
There exists an absolute constant $C_0>0$ such that for any integer $r\ge 2$ and sufficiently large $n$, $\hat{R}(P_n,r)\le C_0r^2\log r\cdot n$.
\end{thm}

This result can be read out from~\cite{Kri16} (Theorem 4 there). An alternative proof has recently been offered by Dudek and Pra\l at~\cite{DP18}, it also relies on a DFS-based argument. The estimate of Theorem~\ref{t-Ramsey} is close to being tight -- Dudek and Pra\l at proved in a prior paper~\cite{DP17} that $\hat{R}(P_n,r)=\Omega(r^2)n$; a somewhat better absolute constant for the lower bound has been obtained by the author~\cite{Kri16}.

\begin{proof}[Proof (sketch)] Let $N=25rn$, and consider the binomial random graph $G(N,p)$ with $p=\frac{Cr\log r}{N}$, where $C>0$ is sufficiently large constant. Then with high probability $G\sim G(N,p)$ has average degree $d=(1+o(1))Np=(1+o(1))Cr\log r$.  In order to bound the typical local density of $G$, observe that for a subset $W\subset [N]$ of cardinality $|W|=i\le 2n$, the number $e_G(W)$ of edges spanned by $W$ in $G$ is a binomial random variable with parameters $\binom{i}{2}$ and $p$. Recalling the easy estimate $\Pr[\Bin(n,p)\ge k]\le \binom{n}{k}p^k\le (enp/k)^k$, we can write:
$$
\Pr\left[e_G(W)\ge \frac{Ci\log r}{9}\right]\le \left[\frac{\frac{ei^2}{2}\cdot\frac{Cr\log r}{N}}{\frac{Ci\log r}{9}}\right]^{\frac{Ci\log r}{9}}= \left(\frac{9er}{2}\cdot\frac{i}{N}\right)^{\frac{Ci\log r}{9}}=\left(\frac{9e}{50}\cdot\frac{i}{n}\right)^{\frac{Ci\log r}{9}}\,.
$$
Summing over all $i\le 2n$ and over all choices of $W\subset [N]$ with $|W|=i$, performing standard asymptotic manipulations, and recalling that $C>0$ can be chosen to be sufficiently large, we derive that with high probability every such $W$ spans at most $(d/8r)|W|$ edges.

Let now $f\colon E(G)\rightarrow [r]$ be an $r$-coloring of $E(G)$. Take $E_0$ to be the majority color, clearly $|E_0|\ge |E|/r$. Applying Proposition~\ref{p-Ramsey} to $G$ and $E_0$, we conclude that $E_0$ contains a path of length~$n$. It follows that every $r$-coloring of $E(G)$ contains a monochromatic path of length~$n$. Hence $\hat{R}(P_n,r)\le |E(G)|=O(r^2\log r)n$.
\end{proof}

\subsection{Long cycles}\label{ss-cycles}
With more effort (still based on the DFS algorithm and its properties) one can prove that every $\alpha$-expander on $n$ vertices contains a cycle of length linear in $n$. Notice that here, unlike in the case of paths, we cannot expect to get a cycle of any prescribed length, but rather a cycle at least as long as the lower bound given by the next theorem.
\begin{thm}[\cite{Kri16}]\label{t-long_cycle}
Let $k>0,\ell\ge 2$ be integers. Let $G=(V,E)$ be a graph on more than $k$ vertices, satisfying:
$$
|N_G(W)|\ge \ell,\quad \mbox{for every } W\subset V,\  \frac{k}{2}\le |W|\le k\,.
$$
Then $G$ contains a cycle of length at least $\ell+1$.
\end{thm}

\begin{proof}[Proof (sketch)]
Observe first that $G$ has a connected component $C$ on more than $k$ vertices. Let $T$ be the tree obtained by applying the DFS algorithm to $C$, under an arbitrary order $\sigma$ of its vertices, and let $r$ be its root (which is the first vertex of $C$ under $\sigma$). Then $|V(T)|=|V(C)|>k$. Now we find a vertex $v$ and a subset $X$ of its children in $T$ such that the subtrees of $T$ rooted at the vertices of $X$ have between $k/2$ and $k$ vertices in total. (This can be done by going from $r$ down the tree till we find a vertex $v$ whose subtree $T_v$ has more than $k$ vertices, but the subtrees $T_x$ of its children in $T$ are all less than $k$ in their sizes; the required set $X$ can then be recruited from the children of $v$ in $T$.)  Let $W=\bigcup_{x\in X}V(T_x)$. Denote by $P$ the path in $T$ from the root $r$ to $v$. Then $k/2\le |W|\le k$, and by Property {\bf (P4)} of the DFS algorithm we have: $N_G(W)\subseteq V(P)$. Let $v^*\in V(P)$ be the farthest from $v$ vertex in $N_G(W)\cap V(P)=N_G(W)$. Clearly its distance from $v$ is at least $|N_G(W)|-1\ge \ell-1$, by the theorem's assumption. Let $w$ be a neighbor of $v^*$ in $W$. Then the cycle, formed by the path in $T$ from $w$ to $v^*$ (which includes the part of $P$ from $v$ to $v^*$) and the edge $(w,v^*)$, has length at least $\ell+1$.
\end{proof}

Assuming that $G$ is an $\alpha$-expander on $n=\Omega(1/\alpha)$ vertices and applying Theorem~\ref{t-long_cycle} to $G$, we derive the existence of a cycle of length more than $\alpha n/4$ in $G$. The order of magnitude in $n$ obtained is obviously optimal; moreover, the order of dependence on $\alpha$ in this bound is optimal as well, due to the same example of $K_{\alpha n,(1-\alpha)n}$.

\subsection{Cycle lengths}\label{ss-cycle_lengths}
Theorem~\ref{t-long_cycle} argues that if $G$ is an $\alpha$-expander on $n$ vertices, then $G$ has a cycle of length $\Omega(\alpha)n$. Also, due to Proposition~\ref{p1}, $G$ has cycles as short as $O_{\alpha}(\log n)$. What can be said then about cycle lengths in $G$ between the two extremes? Does $G$ have linearly many in $n$ cycle lengths? How well can one approximate an integer $\ell$ within the set of cycle lengths in $G$? The answer is given by the following theorem:
\begin{thm}[\cite{FKN18}]\label{t-cycle_lengths}
For every $\alpha>0$ there exist $A,C,n_0$ such that for every $\alpha$-expander $G=(V,E)$ on $n\ge n_0$ vertices and every $\ell\in [C\ln n,n/C]$, the graph $G$ has a cycle whose length is between $\ell$ and $\ell+A$.
\end{thm}

Hence the set $L(G)$ of cycle lengths in an $\alpha$-expander $G$ on $n$ vertices is well spread and approximates every number $\ell$ in the range $[C\ln n,n/C]$ within an additive constant. In particular, $G$ has linearly many cycle lengths.

The proof of Theorem~\ref{t-cycle_lengths} is somewhat involved. We outline instead the proof of a weaker statement, assuring that the set $L(G)$ of cycle lengths in an $\alpha$-expander $G$ on $n$ vertices has cardinality linear in $n$. The proof borrows some ideas from prior papers on cycle lengths, such as~\cite{SV08}.

We start with the following lemma.
\begin{lemma}\label{l-cycle_lengths}
Let $\epsilon>0$. There are constants $C_1=C_1(\alpha,\epsilon)$ and $C_2=C_2(\alpha,\epsilon)>0$ such that for every $v\in V$ the graph $G$ contains a spanning tree $T$ rooted at $v$ with levels $L_0=\{v\},L_1,\ldots,L_{k_0},\ldots, L_{k_1}$ such that $k_0=O(\log n), k_1\le k_0+C_1$,
$$
\left| \bigcup_{i=k_0}^{k_1} L_i \right| \ge (1-\epsilon)n\,,
$$
and the degrees of all $u\in \bigcup_{i=k_0}^{k_1} L_i$ in $T$ are at most $C_2$.
\end{lemma}
The lemma guarantees that one can find a (spanning) tree $T$ in $G$ with few consecutive ``thick" layers $L_i$, $i=O(\log n)$, containing nearly all vertices of $G$ and such that the degrees in $T$ of vertices in these thick layers are all bounded.

Let now $T$ be the tree guaranteed by Lemma ~\ref{l-cycle_lengths} with $\epsilon=\alpha/4$ and arbitrary $v$. Denote $W=\bigcup_{i=k_0}^{k_1} L_i$. Observe that every $U\subset W$, $|U|=n/2$, has at least $\alpha |U|- (n-|W|)\ge \alpha n/4$ neighbors in $W$. Hence $W$ spans a path $P$ of length at least $\alpha n/4$ by Proposition~\ref{p-DFS}.

For each $u\in L_{k_0}$ let $T_u$ be the full subtree of $T$ rooted at $u$. Due to the bounded degree assumption on $T$ we have $|T_u|\le C_2^{C_1}$.

For each $u\in L_{k_0}$ with $T_u\cap V(P)\ne\emptyset$ (there are at least $\frac{|V(P)|}{C_2^{C_1}}=\Theta(n)$ of those) choose a highest (closest to $v$) vertex $x\in T_u\cap V(P)$. Let $X$ be the set of chosen vertices. Now, choose a maximum subset $X_0\subseteq X$ such that the distance along $P$ between any two vertices of $X_0$ is at least $C_1+1$. We have still $|X_0|=\Theta(n)$.

Let $T'$ be the minimal subtree of $T$ whose set of leaves is $X_0$, let $r$ be its root. Due to the minimality $T'$ branches at $r$. Let $A\subset X_0$ be the set of vertices of $X_0$ in one of the branches, and let $B=X_0-A$. We can assume that $|B|\ge |A|$, implying $|B|\ge |X_0|/2$. Let $a$ be an arbitrary vertex of $A$. At least $|B|/2\ge |X_0|/4$ vertices of $B$ are on the same side of $a$ in $P$, denote this set by $B_0$.

For every $b\in B_0$ we get a cycle $C_b$ in $G$ as follows: first, the path from $r$ to $a$ in $T'$, then we move from $a$ to $b$ along $P$, and finally take the path from $b$ to $r$ in $T'$. (These three pieces do not intersect internally, as $a$ and $b$ were chosen as highest vertices of $P$ in their corresponding subtrees $T_u$, $u\in L_{k_0}$.) It is easy to check that the lengths of the cycles $C_b$ are all distinct. Altogether we get $|B_0|=\Theta(n)$ different cycle lengths as promised.

\section{Minors in expanding graphs}\label{s-minors}
Minors is one of the most central notions in modern graph theory. It is thus only natural to expect to see meaningful research connecting expanding graphs and minors. And indeed, there have been several papers, addressing this subject directly or indirectly. We will mention some of them in this section, along with a description of new results due to Rajko Nenadov and the author.

First we recall the definition of a minor.  Let $G=(V,E)$, $H=(U,F)$ be graphs with $U=\{u_1,\ldots, u_k\}$. We say that $G$ contains $H$ as a {\em minor} if there is a collection $(U_1,\ldots,U_k)$ of pairwise disjoint vertex subsets (supernodes) in $V$ such that each $U_i$ spans a connected subgraph in $G$, and whenever $(u_i,u_j)\in F$, the graph $G$ has an edge between $U_i$ and $U_j$. (Then contracting each $U_i$ to a single vertex produces a copy of (a supergraph of) $H$.) Notice that if $G$ contains a minor of $H$ then $|V(G)|\ge |V(H)|$ and $|E(G)|\ge |E(H)|$; these trivial bounds provide an obvious but meaningful benchmark for minor embedding statements.

As we discussed in Section~\ref{s-separators}, there is a very close connection between expanders and graphs without small separators. This connection  enables to extend known results about embedding minors in graphs without small separators to claims about embedding minors in expanding graphs. In particular, it follows from the result of Plotkin, Rao and Smith~\cite{PRS94} that an $\alpha$-expanding graph on $n$ vertices, $\alpha>0$ a constant, contains a minor of the complete graph $K_{c\sqrt{n/\log n}}$. (In fact, the complete minor delivered by the result of~\cite{PRS94} is shallow, i.e., each supernode $U_i$ has diameter $O(\log n)$ in $G$.) An even stronger result has been announced by {Kawarabayashi and Reed, who claimed in~\cite{KR10} that (translating it to the language of expanders) an $\alpha$-expanding graph on $n$ vertices contains a minor of  $K_{c\sqrt{n}}$. Also, Kleinberg and Rubinfeld proved in~\cite{KR96} that an $\alpha$-expander on $n$ vertices  of maximum degree $\Delta$ contains all graphs with $O(n/\log^{\kappa}n)$ vertices and edges as minors, for $\kappa=\kappa(\alpha,\Delta)>1$. Phrasing it differently, a bounded degree $\alpha$-expander is {\em minor universal} for all graphs with $O(n/\log^{\kappa}n)$ vertices and edges. This is optimal up to polylogarithmic factors due to the above stated trivial bound, as there exist $n$-vertex expanders with $\Theta(n)$ edges. (Formally, the number of vertices $n$ is another bottleneck here.) It is worth mentioning here that the bounded maximum degree assumption $\Delta(G)=\Delta=O(1)$ is essential for the argument of Kleinberg and Rubinfeld, as the latter relies heavily on mixing properties of random walks; this issue has been briefly touched upon in Section~\ref{s2}.

We now present recent results, along with outlines of their proofs.

\subsection{Large minors in expanding graphs}\label{ss-minors1}
Our aim here is to present a universality-type result, guaranteeing the existence of a minor of {\em every} graph with bounded (as a function of $n$) number of vertices and edges in every $\alpha$-expanding graph. This theorem is obtained jointly with Rajko Nenadov.

\begin{thm}\label{t-minors1}
For every $\alpha>0$ there exist $c,n_0$ such that the following holds. Let $G$ be a graph on $n\ge n_0$ vertices, and let $H$ be a graph with at most $cn/\log n$ vertices and edges. Then $G$ has a separator of size at most $\alpha n$, or $G$ contains $H$ as a minor.
\end{thm}
(We bound the number of vertices in $H$ as well so as to avoid the pathological case of $H$ having $O(n/\log n)$ edges, but many more vertices, say, more than $n$, thus making the embedding of $H$ as a minor in $G$ impossible.)

 The guarantee $|V(H)|+|E(H)|=O(n/\log n)$ of the above theorem is asymptotically optimal up to multiplicative constants due to the following argument. Let $G$ be an expanding graph of $n$ vertices with logarithmic girth $\girth(G)=\Theta(\log n)$, say, the celebrated Lubotzky-Phillips-Sarnak graph~\cite{LPS88}, and let $H$ be a collection of $k$ vertex-disjoint triangles. If $H$ is a minor of $G$, then the image of every triangle in a minor embedding of $H$ in $G$ contains a cycle, and these cycles are disjoint for distinct triangles in $H$. It follows that $k\cdot \girth (G)\le n$, implying $k=O(n/\log n)$. Hence the largest collection of disjoint triangles that can be embedded in $G$ as a minor counts $O(n/\log n)$ vertices.

\begin{proof}
The proof borrows substantially from the ideas of Plotkin, Rao and Smith~\cite{PRS94}. Just as in many arguments of this sort, we start by assuming for convenience (of the proof) that the target graph $H=(U,F)$ has maximum degree (at most) three. This is possible as by splitting the vertices of $H$ and edges emanating from them, we can construct a graph $H'=(U',F')$ with $\Delta(H')\le 3$ and $|U'|=O(|U|+|F|)$ so that $H$ is a minor of $H'$. Since minor containment is transitive, we can embed instead $H'$ in $G$ as a minor.

Rephrasing Corollary~\ref{coroll2}, we get the following handy lemma:
 \begin{lemma}\label{l-minors}
For every $\alpha>0$ there exists $C_0>0$ such that every graph $G$ on $n$ vertices has diameter at most $C_0\log n$, or contains a subset $U\subset V(G)$ satisfying: $|U|\le n/2$ and $|N_G(U)|\le \alpha n$.
\end{lemma}

The proof of Theorem~\ref{t-minors1} is described as an algorithmic procedure that, given $G$ and $H$ with $|V(H)|=k\le \frac{\alpha n}{6C_0\log n}$ and $\Delta(H)\le 3$ (where $C_0$ is the constant from Lemma~\ref{l-minors}), outputs a small separator  in $G$ or a minor embedding of $H$ in $G$. In fact, this procedure can be easily turned into a polynomial time algorithm performing this task. Let us assume $V(H)=[k]$.  The algorithm maintains and updates a partition $V=A\cup B\cup C$ of the vertex set of $G$, where the set $A$ will eventually contain a large and non-expanding set, witnessing the existence of a small separator in $G$, or the set $B$ will contain the desired minor embedding of $H$; the set $C$ will serve as a vertex reservoir. We will maintain $|N_G(A,C)|\le \alpha |A|$. We will also maintain and update a subset $I_0\subset [k]$, which will describe the current induced subgraph of $H$, minor embedded in $G[B]$. Accordingly, there will be disjoint subsets $B_i\subseteq B$, $i\in I_0$, each spanning a connected subgraph in $G$; we commit ourselves to having always: $|B_i|\le 2C_0\log n$.

We initialize $A=B=\emptyset$, $C=V$, $I_0=\emptyset$.

The procedure will repeat the following loop. Choose an arbitrary $i\in [k]\setminus I_0$. Let $X$ be the set of neighbors of $i$ in $I_0$, $|X|\le 3$. For $j\in X$, choose an arbitrary $v_j\in N_G(B_j,C)$. If for some $j\in X$ there is no such neighbor, we dump $B_j$ into $A$ and update accordingly: $A:=A\cup B_j$, $B:=B\setminus B_j$, $I_0:=I_0\setminus \{j\}$. Otherwise, we apply Lemma~\ref{l-minors} to the induced subgraph $G[C]$. In case it outputs a set $U$, $|U|\le |C|/2$, $|N_G(U,C\setminus U)|\le \alpha |U|$, we move $U$ over to $A$ and update: $A:=A\cup U$, $C:=C\setminus U$. In the complementary case, where $G[C]$ has logarithmic diameter, we find a subset $Y\subset C$ of size $|Y|\le 2\cdot \diam(G[C])\le 2C_0\log n$, such that $G[Y]$ is connected and contains vertices $v_j, j\in X$ (choose one of $v_j$'s as the pivot vertex and connect it by paths of length at most $C_0\log n$ to other $v_j$'s).  We then update: $B:=B\cup Y$, $C:=C\setminus Y$, $B_i:=Y$, $I_0=I_0\cup \{i\}$. (In case $X=\emptyset$, we can simply take $Y$ to be an arbitrary vertex of $C$ and perform the same update.) The set $B_i$ is connected and has an edge connecting it to $B_j$, $j\in X$.

Since we always add to the current $A$ a piece $U$ satisfying $|N_G(U,C)|\le \alpha |U|$, and the set $C$ only shrinks as the algorithm proceeds, we observe that indeed at any point of the execution the current set $A$ satisfies: $|N_G(A,C)|\le \alpha |A|$. Also, the set $B$ is not too large -- it is always composed of sets $B_i$, $i\in I_0$, and thus has at most $k\cdot 2C_0\log n=\alpha n/3$ vertices.

The above procedure eventually reaches the situation when $A$ exceeds $n/3$ or $I_0=[k]$. In the former case, it is easy to see that the first moment the size of $A$ goes above $n/3$, it satisfies: $|A|\le 2n/3$. Also, $|N_G(A)|\le |B|+|N_G(A,C)|\le \alpha n/3+\alpha|A|\le \alpha n/3+2\alpha n/3=\alpha n$. In this case, the set $N_G(A)$ witnesses the existence of a small separator in $G$, separating between $A$ and $V(G)\setminus (A\cup N_G(A)$. In case where the stopping criterion has been reached due to $I_0=[k]$, the set $B$ spans a minor of $H$ in $G$, where the image of $u_i\in V(H)$ is the (connected) set $B_i$.
\end{proof}

\begin{corol}\label{cor-minors}
Let $G$ be an $\alpha$-expander on $n$ vertices. Then $G$ is minor universal for the class of all graphs $H$ with at most $cn/\log n$ vertices and edges, for some $c=c(\alpha)>0$.
\end{corol}

Let us add few remarks here. First, as we have indicated, the proof presented is algorithmic and allows to find in time polynomial in $n$ a minor embedding of a given graph $H$ with $|V(H)|+|E(H)|\le cn/\log n$. Secondly, the minor embedding produced is shallow, i.e., the image of every vertex $u\in V(H)$ in the embedding is a connected set $U$ in $G$, spanning a graph of diameter logarithmic in $n$. Working out the details of the proof, we see that the constant $c$ in Theorem~\ref{t-minors1}, and thus in Corollary~\ref{cor-minors} depends quadratically on $\alpha$. Finally, notice that the argument above can be trivially adapted to the case where $\alpha=\alpha(n)$ is vanishing.

\subsection{Complete minors in expanding graphs}\label{ss-minors2}
Theorem~\ref{t-minors1} guarantees that every expanding graph $G$ on $n$ vertices contains every graph $H$ with $O(n/\log n)$ vertices and edges as a minor. As we argued in the previous subsection, this estimate on the size of $H$ is optimal. For some particular types of $H$ however we can do better. For example, due to Proposition~\ref{p-DFS} an expanding $n$-vertex graph $G$ contains a path of length linear in $n$. One can also prove easily that such $G$ contains a linearly sized star as a minor. (First find a linearly long path $P$ in $G$, then argue that due to expansion one can find in $G$ a matching $M$ of linear size with every edge intersecting $V(P)$ in exactly one vertex; contracting $V(P)$ to a single vertex produces the required star.)

A particularly prominent class of graphs to try and find as minors is cliques. How large a clique minor can one hope to find in a weak expander? Denote by $ccl(G)$ the clique contraction number of $G$, which is the order of a largest clique minor to be found in $G$.  Recall that if $H$ is a minor of $G$, then $|E(H)|\le |E(G)|$. Since weak expanders  on $n$ vertices can have linearly few in $n$ edges, the best bound one can hope to get is $ccl(G)=\Theta\left(\sqrt{n}\right)$. Theorem~\ref{t-minors1} comes quite close and guarantees that a graph $G$ on $n$ vertices without separators of small linear size has $ccl(G)=\Omega\left(\sqrt{n/\log n}\right)$.

The above mentioned result of Kawarabayashi and Reed~\cite{KR10} improves this estimate, proving that a graph $G$ on $n$ vertices without sublinear separators satisfies: $ccl(G)=\Omega\left(\sqrt{n}\right)$.  In fact, their result is much more general as it proves the following alternative: for any $h(n)$, an $n$-vertex graph $G$ has a minor of the complete graph $K_h$, or a separator of order $O(h\sqrt{n})$.

Here we present a proof of the following theorem, a joint work with Rajko Nenadov.
\begin{thm}\label{t-minors2}
For every $\alpha>0$ there exist $c,n_0$ such that the following holds. Let $G$ be a graph on $n\ge n_0$ vertices and maximum degree $d$. Assume that the edge isoperimetric number $i(G)$ satisfies: $i(G)\ge \alpha d$. Then $ccl(G)\ge c\sqrt{n}$.
\end{thm}
Here the edge isoperimetric number $i(G)$ of  a graph $G=(V,E)$ is defined as:
$$
i(G)=\min_{\substack{U\subset V\\ |U|\le |V|/2}}\frac{e_G(U,V\setminus U)}{|U|}\,.
$$
This quantity is fairly similar to the Cheeger constant $h(G)$ (to the extent that it is sometimes called the Cheeger constant of $G$); they are in fact exactly equal up to the factor $d$ for the case of $G$ being a $d$-regular graph.

The above theorem is superseded by the much more general result of Kawarabayashi and Reed. However, the proof of the latter is quite involved, and the full version of the conference paper~\cite{KR10} is still waiting to be published. We thus believe that presenting a reasonably short proof of Theorem~\ref{t-minors2} -- whose outline will come shortly -- certainly has merit.

Theorem~\ref{t-minors2} implies in particular that an $\alpha$-expanding graph on $n$ vertices with bounded degrees has a complete minor of the order of magnitude $\sqrt{n}$. The assumption about edge expansion being comparable to the maximum degree is an artifact of the proof and its techniques (specifically, the use of random walks), and can perhaps be lifted.

\begin{proof}
The proof is somewhat similar to the proof of Theorem~\ref{t-minors1}, and uses some ideas from the arguments of Plotkin, Rao and Smith~\cite{PRS94}, and of Krivelevich and Sudakov~\cite{KS09}. Basically, we construct a complete minor of $K_k$ in $G$, for an appropriately chosen $k$, a supernode by a supernode. Each time, when embedding the next supernode $B_i$ in $G$, we  need to make sure that $B_i$ both spans a connected graph in $G$, and contains a neighbor of every presently embedded supernode $B_{j}$; due to expansion we can assume that the set of neighbors $U_j$ of $B_j$ is reasonably large. We will aim to do it economically, i.e., to get a relatively small set $B_i$, so as to allow enough room for embedding of $k$ supernodes. A fairly natural idea for finding such a set, already suggested in~\cite{KS09}, is to use random walks, and to take $B_i$ to be the trace of a long enough random walk $W$, hoping it will hit each of the neighborhoods $U_j$. Our hopes are supported by the expansion properties of $G$, enabling to argue that $W$ behaves similarly to a random set of the same size in terms of its hitting properties. In reality, it turns out to be more beneficial to take $B_i$ to be the trace of $W$ along with shortest paths from $W$ to each $U_j$, arguing that typically the addition of these shortest paths does not add much to the size of $W$.

Now we get to the actual proof. We start it by introducing a basic tool based on random walks. Recall that a \emph{lazy random walk} on a graph $G = (V, E)$ with the vertex set $V = \{1, \ldots, n\}$ is a Markov chain whose matrix of transition probabilities $P = P(G) = (p_{i,j})$ is defined by
\[
	p_{i,j} = \begin{cases}
		\frac{1}{2\deg_G(i)}, &\text{if } \{i,j\} \in E(G) \\
		\frac{1}{2}, &\text{if } i = j, \\
		0, &\text{otherwise.}
	\end{cases}
\]
This Markov chain has the stationary distribution $\pi$ defined as $\pi(i) = \deg_G(i) / 2e(G)$. As usual, for a subset $A\subseteq [n]$, we write $\pi(A)=\sum_{i\in A} \pi(i)$.

The following lemma gives an upper bound on the probability that a lazy random walk avoids a given subset $U \subseteq V(G)$.

\begin{lemma} \label{lemma:random_walk_miss_U}
	Let $G$ be a graph with $n$ vertices and maximum degree $d = d(n)$. Then for any $U \subseteq V(G)$ the probability that a lazy random walk on $G$, which starts from the stationary distribution $\pi$ and makes $\ell$ steps, does not visit $U$ is at most
	\[
		\exp\left(- \frac{i(G)^3}{8d^3} \cdot \frac{|U|\ell}{n}\right).
	\]		
\end{lemma}

We now describe the proof of Lemma~\ref{lemma:random_walk_miss_U}.

Let $1=\lambda_1 \ge \lambda_2 \ge \ldots \ge \lambda_n$ be the eigenvalues of the transition matrix $P$. The \emph{spectral gap} of $P$ (or of $G$) is defined as $\eta(G) =  1 - \lambda_2$. The following result of Mossel et al.~\cite{MORSS06} (more precisely, the first case of~\cite[Theorem 5.4]{MORSS06}) relates the spectral gap to the probability that a random walk does not leave a specific subset. We state its version tailored for our needs.

\begin{thm}[\cite{MORSS06}] \label{thm:spectral}
	Let $G$ be a connected graph with $n$ vertices and let $\eta(G)$ be the spectral gap of the transition matrix $P = P(G)$. Then the probability that a lazy random walk of length $\ell$, starting from a vertex chosen according to the stationary distribution $\pi$, does not leave a non-empty subset $A \subseteq V(G)$ is at most
	\[
		\pi(A) (1 - \eta(G)(1 - \pi(A)))^\ell.
	\]		
\end{thm}

In order to bound the spectral gap of $G$, we use the celebrated result of Jerrum and Sinclair~\cite[Lemma 3.3]{JS89}, relating the spectral gap of $P(G)$ to its \emph{conductance} $\Phi(G)$, defined as
\[
	\Phi(G) = \min_{\substack{S \subseteq V \\ 0 < \pi(S) \le 1/2}} \frac{\sum_{i \in S, j \notin S} \pi(i) p_{i,j}}{\pi(S) }.
\]

\begin{lemma}[\cite{JS89}] \label{lemma:conductance}
	Let $G = (V, E)$ be a connected graph. Then the second eigenvalue $\lambda_2$ of the transition matrix $P$ satisfies:  $1-\lambda_2 \ge \Phi(G)^2/2$.
\end{lemma}

Finally, it is easy to verify that $\Phi(G)\ge i(G)/2d$, where as before $i(G)$ is the edge isoperimetric number of $G$, and $d$ is its maximum degree.

\begin{proof}[Proof of Lemma~\ref{lemma:random_walk_miss_U}]
Consider a non-empty subset $U \subset V(G)$. Theorem~\ref{thm:spectral} states that a lazy random walk never leaves the set $A = V(G) \setminus U$ (i.e., never visits $U$) with probability at most
	\begin{align}
		\pi(A)(1 - \eta(G)(1 - \pi(A)))^{\ell} &= (1 - \pi(U))(1 - \eta(G) \pi(U))^{\ell} \nonumber \\
		&\le \exp(- \eta(G) \ell \pi(U)) \le \exp\left(- \frac{i(G)^2}{8 d^2} \cdot \ell \pi(U)\right), \label{eq:U}
	\end{align}
	where in the last inequality we used Lemma~\ref{lemma:conductance} and $\Phi(G) \ge i(G)/2d$.  From the trivial bound $2e(G) \le d n$ and $\delta(G) \ge i(G)$ we get
	\[
		\pi(U) \ge \frac{|U| i(G)}{2e(G)} \ge |U| \frac{i(G)}{dn},
	\]		
	which after plugging into \eqref{eq:U} gives the desired probability that a random walk misses $U$.
\end{proof}

The next lemma implements our promise to utilize (extensions of) random walks as connected hitting sets.
\begin{lemma} \label{lemma:covering}
	For every $\beta>0$ there exist positive $C = \Theta(1/\beta^3)$ and $n_0$ such that the following holds. Let $G$ be a graph with $n \ge n_0$ vertices, maximum degree $d$, and $i(G) \ge \beta d$. Given $k$ and $s$ such that $ks \ge 2 n$, and subsets $U_1, \ldots, U_k \subseteq V(G)$, where each $U_j$ is of size $|U_j| \ge s$, there exists a connected set $Y \subseteq V(G)$ of size at most
	\[
		|Y| \le C \cdot \frac{n}{s} \ln \left( \frac{k s}{n} \right)
	\]
	which intersects every $U_j$.
\end{lemma}

We mention in passing that for many pairs $(k(n),s(n))$, by choosing subsets $U_i\subset [n]$, $|U_i|=s$, at random one can see that there is a familiy $\{U_i\}_{i=1}^k$,
whose covering number has order of magnitude $(n/s)\log(ks/n)$. Thus Lemma~\ref{lemma:covering} delivers a nearly optimal promise of the size of a hitting set, with an additional -- and potentially important -- benefit of this set spanning a connected subgraph in $G$.

\begin{proof}
	Let
	\[
		\ell = \frac{8}{\beta^3} \cdot \frac{n}{s} \ln \left( \frac{ks}{n} \right)\,,
	\]		
	and consider a lazy random walk $W$ in $G$ which starts from the stationary distribution and makes $\ell$ steps. The desired connected subset $Y$ is now constructed by taking the union of $W$  with a shortest path between $U_j$ and $W$, for each $j \in [k]$. We argue that with positive probability $Y$ is of required size.

For $1\le j\le k$, let $X_j$ be the random variable measuring the distance from $U_j$ to $W$ in $G$. Then the set $Y$ has expected size at most $\ell+\sum_{j=1}^k E[X_j]$.

In order to estimate the expectation of $X_j$, for a positive integer $z$ write $p_{j,z}=Pr[X_j=z]$ and $p_{j,\ge z}=Pr[X_j\ge z]$. Trivially $p_{j,z}=p_{j,\ge z}-p_{j,\ge z+1}$. Hence
$$
E[X_j]=\sum_{z\ge 1} zp_{j,z}=\sum_{z\ge 1} z\left(p_{j,\ge z}-p_{j,\ge z+1}\right)\le \sum_{z\ge 1}p_{j,\ge z}\,.
$$

Let $U_{j,z}$ be the set of vertices of $G$ at distance at most $z$ from $U_j$. The graph $G$ is easily seen to be a $\beta$-expander, and similarly to Proposition~\ref{p1} we have: $|U_{j,z}|\ge \min\left\{s(1+\beta)^z,\frac{n}{2}\right\}$. Applying Lemma~\ref{lemma:random_walk_miss_U} to $U_{j,z}$, we obtain for $z\le \log(n/2)/\log(1+\beta)$:
$$
p_{j,\ge z}\le \exp\left\{-\frac{\beta^3(1+\beta)^zs\ell}{8n}\right\}\,.
$$
Plugging in this estimate, recalling the value of $\ell$, and doing fairly straightforward arithmetic, we can derive:
$$
E[X_j]= \sum_{z\ge 1}p_{j,\ge z}=O\left(\left(\frac{n}{ks}\right)^{1+\beta}\right)\,.
$$
Hence
$$
E[|Y|]=\ell+\sum_{j=1}^k E[X_j]=\ell+O\left(k\left(\frac{n}{ks}\right)^{1+\beta}\right)=O(\ell)\,.
$$
\end{proof}

We are now ready to complete the proof of Theorem~\ref{t-minors2}. Let $C$ be a constant given by Lemma~\ref{lemma:covering} with $\beta=\alpha/4$, and set
	\[
		b = \sqrt{\frac{8 C}{\alpha} \cdot n} \qquad \text{ and } \qquad k = \frac{\alpha n}{6b}=\Theta(\alpha^3)\sqrt{n}\,.
	\]

The argument here is somewhat similar to the proof of Theorem~\ref{t-minors1}. We maintain and update a partition $V(G)=A\cup B\cup C$ so that always $e_G(A,C)< \beta d|A|$. There is also a set of indices $I_0\subseteq [k]$ (corresponding to the set of supernodes $B_i,i\in I_0$), such that the subsets $B_i\subset B$ are pairwise disjoint, span connected subgraphs in $G$, and $|B_i|=b$. Moreover, if $i_1\ne i_2\in I_0$, then $G$ has an edge between $B_{i_1}$ and $B_{i_2}$.

Initially $A=B=\emptyset$, $C=V$, $I_0=\emptyset$. As long as $|A|\le n/3$ or $I_0\subsetneq [k]$, we repeat the following loop. If there exists $i\in I_0$ such that $e_G(B_i,C)< \beta d|B_i|$, we move $B_i$ to $A$ and update: $A:=A\cup B_i$, $B:=B\setminus B_i$, $I_0:=I_0\setminus\{i\}$. Otherwise, recalling that $\Delta(G)\le d$, we derive that $|N_G(B_i,C)|\ge \beta b$ for all $i\in I_0$.

Now look into the induced subgraph $G[C]$. If there is a subset $U\subset C$, $|U|\le |C|/2$ with $e_G(U,C\setminus U)<\beta d|U|$, we move $U$ to $A$ and update $A:=A\cup U$, $C:=C\setminus U$. Otherwise, the subgraph $G[C]$ satisfies $i(G[C])\ge \beta d$. We now apply to it  Lemma~\ref{lemma:covering} with $s=\beta b$ and the sets $U_j=N_G(B_j,C)$, $j\in I_0$ (adding dummy sets $U_j$ to get $k'$ sets $U_j$ altogether, all of size at least $s$, so that $2n<k's<3n$) to get a connected set $Y$ of size
$|Y| \le C\frac{n}{s}\ln 3<b$. Finally we extend  $Y$ to a connected set $Y'$ of exactly $b$ vertices in $C$, choose an arbitrary $i\in [k]\setminus I_0$ and update: $B:=B\cup Y'$, $C:=C\setminus Y'$, $B_i:= Y$, $I_0:=I_0\cup\{i\}$.

If the procedure terminated due to $I_0 = [k]$, then we have found the required $K_k$-minor. Suppose that this is not the case. We have $|B|<bk=\alpha n/6$. In the last step the set $A$ has become larger than $n/3$, but still $|A| \le 2n/3$ . This implies $\min(|A|, |V \setminus A|) \ge n/3$, and hence $e_G(A, V \setminus A) \ge i(G)n / 3\ge \alpha nd/3$. On the other hand, $e_G(A,B)\le d|B|\le \alpha nd/6$, and $e_G(A,C)<\alpha d/4\cdot |A|\le \alpha dn/6$. Altogether $e_G(A,V\setminus A)<e_G(A,B)+e_G(A,C)< \alpha nd/6+\alpha nd/6=\alpha nd/3$. The obtained contradiction shows that the algorithm always runs to a successful end, outputting a minor of $K_k$ in $G$.
\end{proof}

Let us mention that \cite{KN19} presents a result about large complete minors in expanding graph with growing degrees. Since the result's statement is somewhat involved and assumes stronger edge expansion of small sets, we decided not to state it here, referring the interested reader to \cite{KN19} for details.  

\subsection{Large minors in random graphs}\label{ss-minors3}
We conclude this section by presenting consequences of the above results for random graphs.

As we argued in Section~\ref{s-subgraphs} (see Corollary~\ref{cor-sparse} there), a supercritical random graph $G\sim G(n,c/n)$, $c>1$, typically contains an expanding subgraph $G^*$ of bounded degree on linearly many vertices. Applying Corollary~\ref{cor-minors} and  Theorem~\ref{t-minors2} to $G^*$, we derive the following corollaries.

\begin{coroll}\label{cor-minors2}
For every $c>1$ there exists $\delta>0$ such that a random graph $G\sim G(n,c/n)$ is with high probability minor universal for the class of all graphs $H$ with at most $\delta n/\log n$ vertices and edges.
\end{coroll}

\begin{coroll}\label{cor-minors3}
For every $c>1$ there exists $\delta>0$ such that a random graph $G\sim G(n,c/n)$ contains with high probability a minor of the complete graph on $\lfloor\delta\sqrt{n}\rfloor$ vertices.
\end{coroll}

Both corollaries provide an asymptotically optimal order of magnitude. For Corollary~\ref{cor-minors2}, notice that  $G\sim G(n,c/n)$ contains typically $o(n/\log n)$ cycles of length at most $c_0\log n$, for some $c_0=c_0(c)>0$ small enough (straightforward first moment argument), and thus we can repeat the argument from Section~\ref{ss-minors1}. For Corollary ~\ref{cor-minors3}, observe that $G\sim G(n,c/n)$ has typically $O(n)$ edges.

Corollary~\ref{cor-minors3} reproves a result of Fountoulakis,  K\"uhn and Osthus~\cite{FKO08}, obtained by applying direct ad hoc methods (and by working quite hard one may add).

Finally, recall that for $d\ge 3$ a random $d$-regular graph $G_{n,d}$ is an $\alpha$-expander for $\alpha=\alpha(d)>0$ (see Section
\ref{s-examples} for discussion). We thus obtain:

\begin{coroll}\label{cor-minors4}
For every integer $d\ge 3$ there exists $\delta>0$ such that a random graph $G\sim G_{n,d}$ contains with high probability a minor of the complete graph on $\lfloor\delta\sqrt{n}\rfloor$ vertices.
\end{coroll}

This has been proven in another paper of Fountoulakis,  K\"uhn and Osthus~\cite{FKO09}, through the use of the so-called configuration model, and a quite substantial contiguity result about random graphs. The argument presented here provides an alternative, and perhaps more conceptual, way to derive this result.

\bigskip

\noindent{\bf Acknowledgement.} The author wishes to thank the anonymous referee for their careful reading and helpful remarks. He is also very grateful to Rajko Nenadov for his cooperation in parts of the research presented in this survey. Finally, the author thanks Limor Friedman, Rajko Nenadov and Wojciech Samotij for  their input and remarks.

\myaddress

\end{document}